\documentclass[a4paper,reqno, 11pt]{amsart}

\usepackage[margin=3cm]{geometry}

\makeatletter
\let\old@setaddresses\@setaddresses
\def\@setaddresses{\bigskip\bgroup\parindent 0pt\let\scshape\relax\old@setaddresses\egroup}
\makeatother

\usepackage{amssymb} 
\usepackage{amsmath} 
\usepackage{amsthm} 
\usepackage[unicode=true]{hyperref}
\hypersetup{
    colorlinks,
    linkcolor={red!50!black},
    citecolor={blue!50!black},
    urlcolor={blue!80!black}
}

\usepackage{todonotes}
\usepackage{bbm}
\usepackage{paralist}
\usepackage[capitalise]{cleveref}

\newtheorem{theorem}{Theorem}[section]
\newtheorem{claim}{Claim}
\newtheorem{question}{Question}

\newtheorem{lemma}[theorem]{Lemma}
\newtheorem{corollary}[theorem]{Corollary}

\newtheorem{observation}[theorem]{Observation}

\theoremstyle{remark}
\newtheorem*{acknowledgement}{Acknowledgments}

\makeatletter
\newcommand{\vast}{\bBigg@{4}}
\newcommand{\Vast}{\bBigg@{5}}
\makeatother
\definecolor{bulgarianrose}{rgb}{0.28, 0.02, 0.03}
\definecolor{gray}{rgb}{0.5, 0.5, 0.5}

\makeatletter
\def\namedlabel#1#2{\begingroup
    #2%
    \def\@currentlabel{#2}%
    \phantomsection\label{#1}\endgroup
}
\usepackage[normalem]{ulem}
\usepackage{dsfont}
\usepackage{accents}
\usepackage{verbatim}
\usepackage{ulem}
\usepackage{pgf,tikz,pgfplots}
\pgfplotsset{compat=1.16}
\usepackage[all]{xy} 

\usepackage{stackengine}
\stackMath
\newcommand\tsup[2][2]{%
 \def\useanchorwidth{T}%
  \ifnum#1>1%
    \stackon[-.5pt]{\tsup[\numexpr#1-1\relax]{#2}}{\scriptscriptstyle\sim}%
  \else%
    \stackon[.5pt]{#2}{\scriptscriptstyle\sim}%
  \fi%
}

\newcommand{\lbox}{\mathrm{lbox}}
\newcommand{\ad}{\mathrm{ad}}
\newcommand{\boxi}{\mathrm{box}}

\usepackage{subfig}

\usepackage{enumerate}

\title{Local Boxicity}

\author[L.~Esperet]{Louis Esperet}
\address[L.~Esperet]{Laboratoire G-SCOP (CNRS, Univ.\ Grenoble Alpes), Grenoble, France}
\email{louis.esperet@grenoble-inp.fr}

\author[L.~Lichev]{Lyuben Lichev}
\address[L.~Lichev]{Univ. Jean Monnet, Saint Etienne and Institut Camille Jordan,
Lyon, France}
\email{lyuben.lichev@univ-st-etienne.fr}

\thanks{L.\ Esperet is partially supported by the French ANR Projects GATO (ANR-16-CE40-0009-01), GrR (ANR-18-CE40-0032), and by LabEx PERSYVAL-lab (ANR-11-LABX-0025). }

\begin{document}

\maketitle
 
\begin{abstract}
A \emph{box} is the cartesian product of real
intervals, which are either bounded or equal to $\mathbb{R}$. A box is said to be
\emph{$d$-local} if at most $d$ of the intervals are bounded. 
In this paper, we investigate the recently
introduced \emph{local boxicity} of a graph $G$, which is the minimum
$d$ such that $G$ can be
represented as the intersection of $d$-local boxes in some dimension. We prove that
all graphs of maximum degree $\Delta$ have local boxicity $O(\Delta)$, while
almost all graphs of maximum degree $\Delta$ have local boxicity
$\Omega(\Delta)$, improving known upper and lower bounds. We also give
improved bounds on the local boxicity as a
function of the number of edges or the genus. Finally, we investigate
local boxicity through the lens of chromatic graph theory. We prove that
the family of
graphs of local boxicity at most 2 is \emph{$\chi$-bounded}, which means that
the chromatic number of the graphs in this class can be bounded by a
function of their clique number. This extends a classical  result on graphs of
boxicity at most 2.
\end{abstract}



\section{Introduction}

In this paper, a \emph{real interval} is either a bounded real interval, or equal to $\mathbb{R}$. A \emph{$d$-box} in $\mathbb{R}^d$ is the cartesian product of $d$ real intervals. A \emph{$d$-box representation} of a graph $G$ is a collection of $d$-boxes $(B_v)_{v\in V(G)}$ such that for any two vertices $u,v\in V(G)$, $u$
and $v$ are
adjacent in $G$ if and only if $B_u$ and $B_v$ intersect.
The \emph{boxicity} of a graph $G$, denoted by $\boxi(G)$, is the minimum integer $d$ such
that $G$ has a $d$-box representation\footnote{We remark that allowing
intervals of the type $(-\infty, a)$ and $(a, +\infty)$ does not
modify the definition of boxicity, but it will be more convenient to
restrict ourselves to real intervals as defined above (that are either
bounded or equal to $\mathbb{R}$) when we study
local boxicity.}. This parameter was introduced by Roberts~\cite{Rob}
in 1969, and has been extensively studied.

In this paper we investigate the following local variant
of boxicity, which was recently introduced by Bl{\"a}sius, Stumpf and
Ueckerdt~\cite{BSU} (see also~\cite{KU} for a more general framework
on local versions of graph covering parameters). A \emph{$d$-box} in
$\mathbb{R}^{d}$ is \emph{local} in dimension $i$ if its projection on
dimension $i$ is a bounded real interval (equivalently, if the
projection is distinct from $\mathbb{R}$). A \emph{$d$-local box} in
some dimension $d'\ge d$ is a $d'$-box that is local in at most $d$ dimensions. A
\emph{$d$-local box representation} of a graph $G$ in some dimension
$d'$ is  a collection 
$(B_v)_{v\in V(G)}$ of $d$-local boxes in dimension $d'$ such that for any two vertices $u,v\in V(G)$, $u$
and $v$ are
adjacent in $G$ if and only if $B_u$ and $B_v$ intersect.
The
\emph{local boxicity} of a graph $G$, denoted by $\lbox(G)$, is the minimum integer $d$ such
that $G$ has a $d$-local box representation in
some dimension.

It directly follows from the definition that for any graph $G$,
$\lbox(G)\le \boxi(G)$. This raises two types of natural questions:
(1) can we improve known upper bounds (and extend known lower bounds) on the boxicity to the local
boxicity, and (2) can we extend known properties of graphs of boxicity
at most $d$ to graphs of local boxicity at most $d$?

In this paper we prove several results along these lines. The boxicity
of graphs of maximum degree at most $\Delta$ has been extensively studied in the
literature~\cite{ABC,AC,CFS,Esp,SW}, with currently almost matching bounds of $O(\Delta
\log^{1+o(1)}\Delta)$~\cite{SW} and $\Omega(\Delta \log
\Delta)$~\cite{CFS,EKT}. It was proved by Majumder and
Mathew~\cite{MM} in 2018 that graphs of maximum degree $\Delta$ have
local boxicity at most $2^{9\log^* \Delta}\Delta$, where $\log$
denotes the binary logarithm and $\log^*$
denotes the binary iterated logarithm, defined as follows:
\[
    \log^* t = \min\{k\in \mathbb N\hspace{0.2em}|\hspace{0.2em} \underbrace{\log\log\dots\log}_{k \text{ times}} t \le 1\}.
\]

  They also deduced from a result of
\cite{KMMSSUW} that there are graphs with maximum degree $\Delta$ and
local boxicity $\Omega(\Delta/\log \Delta)$. Here we improve these two
bounds as follows. We say that \emph{almost all} graphs of some class
$\mathcal{C}$ satisfy some property $P$ if the proportion of
$n$-vertex graphs of $\mathcal{C}$ satisfying $P$ (among all
$n$-vertex graphs of $\mathcal{C}$) tends to 1 as $n\to \infty$.

\begin{theorem}\label{thm:main}
There are constants $C_1,C_2 > 0$ such that the following holds for every integer $\Delta$.
Every graph of maximum degree $\Delta$ has local boxicity at most
$C_1\Delta$, while almost all graphs of maximum degree $\Delta$ have
local boxicity at least $C_2\Delta$.
\end{theorem}

Using a connection between the  (local) dimension of partially ordered sets
and the (local) boxicity of graphs~\cite{ABC,MM}, \cref{thm:main} directly implies
that partially ordered sets whose comparability graphs have maximum
degree $\Delta$ have local dimension $O(\Delta)$, and this bound is
optimal up to a constant factor (see~\cite{DFGKLNU} for more details and
results on the local dimension of posets).

\medskip

Majumder and
Mathew~\cite{MM}  proved that every graph on $m$ edges has local
boxicity at most $(2^{9\log^* \sqrt{m}}+1)\sqrt{m}$. By adapting their
argument and using Theorem~\ref{thm:main} we improve their bound as
follows.


\begin{theorem}\label{thm:edges}
There is a constant $C_3>0$ such that  every graph with $m$ edges has
local boxicity at most $C_3
\sqrt{m}$. On the other hand, in the regime $m=\Theta(n^2)$, almost
all $n$-vertex graphs with $m$ edges have local boxicity 
$\Omega(\sqrt{m}/\log m)=\Omega(n/\log n)$. 
\end{theorem}

It was proved by Majumder and
Mathew~\cite{MM} that $n$-vertex graphs have local boxicity $O(n/\log n)$, so \cref{thm:edges} implies that almost all $n$-vertex graphs with $m=\Theta(n^2)$ edges have local boxicity 
$\Theta(n/\log n)$.

\medskip

\cref{thm:main} and \cref{thm:edges} will be proved in
Section~\ref{sec:deg}.

\medskip

For a sequence of probability spaces $(\Omega_n, \mathcal F_n, \mathbb P_n)_{n\geq 1}$ and a sequence of events $(A_n)_{n\geq 1}$, where $A_n\in \mathcal F_n$ for every $n\geq 1$, we say that $(A_n)_{n\geq 1}$ happens \textit{asymptotically almost surely} or \textit{a.a.s.}\ if $\underset{n\to +\infty}{\lim}\mathbb P_n(A_n) = 1$.

For any $p\in [0,1]$ and $n\in \mathbb N$,
the random graph $G\in \mathcal G(n,p)$ is defined from the empty
graph on $n$ vertices by adding an edge between each pair of vertices
with probability $p$ and independently from all other pairs.

\medskip

It was proved in \cite{ACS14} that for any $p<1-\tfrac{40 \log n}{n^2}$, $G\in \mathcal G(n,p)$ has boxicity $\Omega(np(1-p))$ a.a.s., and in particular if $p<1-\epsilon$ for some constant $\epsilon>0$, then the boxicity is a.a.s. $\Omega(np)$. Recall that $n$-vertex graphs have local boxicity $O(n/\log n)$, so there is a clear difference between the two parameters in the dense regime (when $p$ is a constant, independent of $n$).

\medskip

For any $\varepsilon > 0$, if $np < 1-\varepsilon$, the random graph $G\in \mathcal G(n,p)$ consists of a set of connected components, each containing at most one cycle a.a.s., and as such has local boxicity at most two (see \cref{lem 2.10}). 
Our next result concerns the local boxicity of the random
graph $\mathcal G(n,p)$ in the sparse regime $np\in [1-\varepsilon, n^{1-\varepsilon}]$. In this regime we prove that the local boxicity is a.a.s. $\Theta(np)$, which in particular implies the lower bound of $\Omega(np)$ on the boxicity of the random graph $G(n,p)$, obtained in \cite{ACS14}.

\begin{theorem}\label{thm:gnp}
Let $\varepsilon\in (0,1)$ and $G\in \mathcal G(n,p)$ with $np\in
[1-\varepsilon, n^{1-\varepsilon}]$. Then there is a constant
$C_4=C_4(\varepsilon)$ such that asymptotically almost surely
\[
    \tfrac{\varepsilon}{41}\cdot np\le \lbox(G)\le C_4 \cdot np.
\]
\end{theorem}

The proof of \cref{thm:gnp} will be given in Section~\ref{sec:RG}.

\medskip

It was proved by Thomassen in 1986 that planar graphs have boxicity
at most 3~\cite{Tho}, which is best possible (as shown by the planar
graph
obtained from a complete graph on 6 vertices by removing a perfect
matching~\cite{Rob}). It is natural to
investigate how this result on planar graphs extends to graphs
embeddable on surfaces of higher genus. It was proved in~\cite{SW}
that graphs embeddable on surfaces of Euler genus $g$ have boxicity
$O(\sqrt{g \log g})$, which is best possible~\cite{Esp2}.
We prove the following version for local boxicity.

\begin{theorem}\label{thm:genus}
For every integer $g\ge 0$, every graph embeddable on a surface of
Euler genus at most $g$ has local boxicity at most $(70+o_g(1))\sqrt{g}$. On
the other hand, there exist a constant $C_5>0$ (independent of $g$) such that some of these graphs have boxicity at least $C_5\sqrt{g}/\log g$.
\end{theorem}

\cref{thm:genus} will be proved in Section~\ref{sec:genus}. The
$\emph{girth}$ of a graph $G$ is the length of a smallest cycle in $G$
(if $G$ is acyclic, it girth is infinite). We now turn to graphs $G$
whose complement, denoted by $G^c$, has girth at least 5. In this
class of very dense graphs we can determine the local boxicity fairly
accurately. Given a graph $G=(V,E)$, the \emph{average degree} of $G$,
denoted by $\ad(G)$, is defined as $2|E|/|V|$.

\begin{theorem}\label{thm:avgdeg}
If $G$ is a graph such that $G^c$ has girth at least five, then $\lbox(G)\ge \left\lfloor\tfrac{\ad(G^c)}{2} + 1\right\rfloor$.
\end{theorem}

In the special case of regular graphs, we give a more precise bound.

\begin{theorem}\label{thm:gcreg}
Let $G$ be a graph and let $k\in \mathbb N$. Suppose that $G^c$ is $k$-regular and has girth at least five. Then,
\begin{itemize}
    \item if $k$ is even, or if $k$ is odd and $G^c$ contains a perfect matching, then $\lbox(G) = \left\lfloor\tfrac{k}{2}+1\right\rfloor$, and
    \item if $k$ is odd and $G^c$ does not contain a perfect matching, then $\lbox(G) = \tfrac{k+3}{2}$.
\end{itemize}
\end{theorem}

A classical result is that the graph obtained from $K_{n}$ ($n$ even) by removing a perfect matching has boxicity $n/2$~\cite{Rob}, while \cref{thm:gcreg} shows that the same graph has local boxicity 1. More generally, if $G$ is a non-complete regular graph on $n$ vertices, whose complement has girth at least 5, then $G$ has boxicity at least $n/4$~\cite{ACS14}, while \cref{thm:gcreg} shows that the local boxicity of $G$ can be very low. This shows a major difference between boxicity and local boxicity.
Note that regular graphs of odd degree, girth at least five and no
perfect matching exist, see \cite{CDR} and \cite{FH}. This shows that
the second case in~\cref{thm:gcreg} cannot be neglected in our
analysis. \cref{thm:avgdeg} and \cref{thm:gcreg} will be proved in Section~\ref{sec:gc}.

\medskip

We conclude the paper with a study of some connections between local
boxicity and graph coloring. We recall that the chromatic number
$\chi(G)$ of a graph $G$ (the minimum number of colors in a proper
coloring of $G$) is at least the \emph{clique number} $\omega(G)$ of
$G$ (the maximum number of pairwise adjacent vertices in $G$). It is
well known that there exist graphs with bounded clique number and arbitrary
large chromatic number, and classes that do not contain such graphs are fundamental objects of study. We say that a class $\mathcal{C}$ is
\emph{$\chi$-bounded} if there is a function $f$ such that for any graph
$G\in \mathcal{C}$, $\chi(G)\le f(\omega(G))$ (see~\cite{SSchi} for a
recent survey on $\chi$-boundedness).

It is well known that the class of graphs of boxicity
at most 2 (also known as \emph{rectangle graphs}) is
$\chi$-bounded~\cite{CW,AG,Hen}, while there are triangle-free graphs
of boxicity 3 and unbounded chromatic number~\cite{Bur}.
Since the class of graphs of boxicity at most 2 is contained in the
class of graphs of local boxicity at most 2, a natural question is
whether this larger class is still $\chi$-bounded. We prove that this
is indeed the case (as the asymptotic complexity of our
$\chi$-bounding function is certainly far from optimal, we make no
effort to optimize the multiplicative constant 320).

\begin{theorem}\label{thm:col}
Let $G$ be a graph of local boxicity at most 2 and let
$r=\omega(G)$. Then, $\chi(G)\le 320\cdot r^3\log (2r)$.
\end{theorem}

In addition, we give an improved upper bound in the case of
triangle-free graphs, which is an analogue of a  theorem of Asplund
and Gr\"unbaum \cite{AG} for boxicity (although we do not prove that our bound is sharp).

\begin{theorem}\label{thm:col18}
The chromatic number of any triangle-free graph $G = (V,E)$ with local boxicity at most two is at most 18.
\end{theorem}

\cref{thm:col} and \cref{thm:col18} are proved in
Section~\ref{sec:col}.

\section{Preliminaries}

\subsection{Boxicity and local boxicity}\label{sec:prelbox}

The \emph{intersection} of $d$ graphs $G_i=(V,E_i)$ ($1\le i \le d$) on the
same vertex $V$ is the graph $G=(V,E_1\cap\dots \cap E_d)$. As
observed by Roberts~\cite{Rob},
a graph $G$ has boxicity at most $d$ if and only if $G$ is the
intersection of at most $d$ interval graphs. In one direction, this can be seen by
projecting a $d$-box representation of $G$ on each of the $d$
dimensions, and in the other direction it suffices to take the
cartesian product of the $d$ interval graphs.
Equivalently, the edge-set of $G^c$ (the complement of $G$) can be
covered by at most $d$ \emph{co-interval graphs} (complements of
interval graphs).

Similar alternative definitions exist for the local boxicity. Consider
an interval graph $G$ and an interval representation of $G$ (in which
we allow intervals to be either bounded real intervals or equal to
$\mathbb{R}$). Note that if some vertex $v$ is mapped to $\mathbb{R}$,
then $v$ is \emph{universal} in $G$, which means that $v$ is adjacent
to all the vertices of $G$. This implies the following alternative
definition of local boxicity (see~\cite{BSU}). A graph has local
boxicity at most $d$ if and only if $G$ is the intersection of
$\ell$ interval graphs $G_1,\ldots,G_{\ell}$ (for some $\ell\ge d$), such
that each vertex $v$ of $G$ is universal in all but at most $d$ graphs
$G_i$ ($1\le i \le \ell$). Equivalently, there exist $\ell$
co-interval graphs
$H_1,\ldots,H_\ell$ which are all subgraphs of $G^c$, and such that
$E(G^c)=E(H_1)\cup\cdots\cup E(H_\ell)$ and each vertex of $V(G)=V(G^c)$ is
contained in at most $d$ graphs $H_i$, $(1\le i \le \ell$).

In the remainder of the paper, it will sometimes be useful to consider
these alternative definitions of local boxicity instead of the
original one.

\medskip

Let $G$ be a graph and fix a vertex $v$ of $G$. Let $(B_u)_{u\in G}$ be a
$d$-box representation of $G-v$. By adding one dimension in which $v$ is
mapped to $0$, the neighborhood $N(v)$ of $v$ to $[0,1]$, and the
remaining vertices of $G$ to $1$, we obtain the following.

\begin{observation}\label{obs:g-v}
For every graph $G$ and for every vertex $v\in G$, $\boxi(G)\le
\boxi(G-v)+1$ and $\lbox(G)\le \lbox(G-v)+1$.
\end{observation}

Given a graph $G=(V,E)$ and a subset $S$ of vertices of $G$, we denote
by $G[S]$ the subgraph of $G$ induced by $S$, and by
$G\langle S\rangle$ the graph obtained from $G$ by adding edges
between every two vertices $u,v$ of $G$ that do not both lie in
$S$. Note that the vertex sets of $G$ and $G\langle S\rangle$
coincide. As all the vertices of $G-S$ are universal in $G\langle
S\rangle$, they can be mapped to $\mathbb{R}$ in every dimension
without loss of generality, and thus the following holds.

\begin{observation}\label{obs:angle}
For every graph $G$ and subset $S$ of vertices of $G$, $\lbox(G\langle S\rangle)=  \lbox(G[S])$.
\end{observation}

Given a bipartition $A,B$ of the vertex set of a graph $G$, we denote
by $G\langle A,B\rangle$ the graph obtained from $G$ by adding edges
between any pair $u,v$ of vertices such that $u$ and $v$ are on the
same side of the bipartition (equivalently, by making $A$ and $B$
cliques in $G$).

\begin{observation}\label{obs:angle2}
For every graph $G=(V,E)$ and any bipartition $A,B$ of $V$, we have
$G=G\langle A\rangle \cap G\langle B\rangle\cap G\langle A,B\rangle$ and thus $\lbox(G)\le
\lbox(G[A])+\lbox(G[B])+\lbox(G\langle A,B\rangle)$.
\end{observation}

In the definitions of local boxicity, the dimension $d'$ of the space
or the number $d'$ of interval graphs in the intersection is unbounded
(as a function of $d$). We now observe that we can always assume
without loss of generality that $d'$ is bounded.

\begin{observation}\label{obs:dn}
Every $n$-vertex graph $G$ of local boxicity at most $d$ has a
$d$-local box representation in dimension $d'\le dn$.
\end{observation}

To see this, note that each vertex is universal in all but at most $d$
dimensions, so if there are more than $dn$ dimension, in some
dimension $i$ all vertices are mapped to $\mathbb{R}$. In this case
dimension $i$ can be omitted.

\medskip

We use this simple observation to associate to each $n$-vertex graph $G$ of local boxicity
$d$ a unique binary word as follows. Consider a $d$-local
representation of $G$, in dimension $d'\le dn$. For each vertex $v$ of $G$, we record the
number $d_v$ of dimensions in which $v$ is not universal, and for each
such dimension $i$, we record $i$ and the interval of $v$ in this
representation. We can always assume without loss of generality that the ends of each interval in
an interval representation of an
$n$-vertex interval graph are integers in $[2n]$, so this takes at most
\[\lceil\log
  d\rceil+d\cdot(\lceil\log(dn)\rceil+ 2\lceil\log(2n)\rceil)\le (d+1)\log d+3d \log n+5d+1
\]
bits per vertex. Assuming $d\ge 2$, this is at most $3d \log n
+7d\log d$ bits per vertex, and thus the complete description of $G$ takes at most $nd(3\log n+7 \log d)$
bits in total. This binary word is enough to reconstruct $G$, so this implies
the following.

\begin{observation}\label{obs:counting}
For any integers $n,d\ge 2$, there are at most $2^{nd(3\log n+7\log d)}$ labelled $n$-vertex graphs of local boxicity at most $d$. 
\end{observation}

There are $2^{{n\choose 2}}$ labelled $n$-vertex graphs, so this
immediately implies the following result, which was established
in~\cite{KMMSSUW} (with a different multiplicative constant).

\begin{corollary}\label{cor:allgraphs}
Almost all $n$-vertex graphs have local
boxicity at least  $\tfrac{n}{21\log n}$.
\end{corollary}

\begin{proof}
By \cref{obs:counting}, there are at most $2^{10 nd\log n}$ labelled
$n$-vertex graphs of local boxicity at most $d$, and thus at most
$2^{10n^2/21}=o(2^{{n\choose 2}})$ $n$-vertex labelled graphs of boxicity at most $\tfrac{n}{21\log
  n}$. It follows that almost all $n$-vertex graphs have local
boxicity at least  $\tfrac{n}{21\log n}$.
\end{proof}

It was proved by
Liebenau and Wormald \cite[Corollary 1.5]{LW} that for any $1\le \Delta\le n-2$ there
are at least $(n/e^2\Delta)^{\Delta n/2} = 2^{\Delta n \log (n/e^2\Delta)/2}$
$n$-vertex $\Delta$-regular graphs. We obtain the following
consequence of \cref{obs:counting}.

\begin{corollary}\label{cor:delta}
For any $\epsilon>0$ and any $\Delta(n)=O(n^{1-\epsilon})$, almost all $n$-vertex graphs of maximum degree
$\Delta=\Delta(n)$ have local boxicity at least $\tfrac{\epsilon}{21}\Delta$.
\end{corollary}

\begin{proof}
By \cref{obs:counting}, there are at most $2^{10 nd\log n}$ labelled
$n$-vertex graphs of local boxicity at most $d$, and thus at most
$2^{10n \Delta \epsilon \log n/21}=o(2^{\Delta n \log (n/e^2\Delta)/2})$
$n$-vertex labelled graphs of
boxicity at most $\tfrac{\epsilon \Delta}{21}$. It follows that almost all $n$-vertex graphs of maximum degree
$\Delta=\Delta(n)=O(n^{1-\epsilon})$ have local
boxicity at least  $\tfrac{\epsilon \Delta}{21}$.
\end{proof}

There are 
\begin{equation*}
\dbinom{{n\choose 2}}{m}\ge \exp\big(m\ln(n(n-1)/2m)\big)  
\end{equation*}
$n$-vertex labelled graphs on $m$ edges. A similar proof as that of \cref{cor:allgraphs} and \cref{cor:delta} shows the following.

\begin{corollary}\label{cor:edges}
For any function $m=\Theta(n^2)$,  almost all $n$-vertex graphs with $m$ edges have local
boxicity $\Omega(\sqrt{m}/\log m)$. 
\end{corollary}

In the regime $m=\Theta(n^2)$,
graphs have Euler genus $g=\Theta(n^2)$
(this follows from Euler's formula, which easily
implies that for any $n$-vertex graph with $m$-edges and Euler genus
$g$, we have $m/3-n+2\le g \le m-n+1$). This implies the following.

\begin{corollary}\label{cor:genus}
For any function $g=\Theta(n^2)$,  almost all $n$-vertex graphs with Euler genus $g$  have local
boxicity $\Omega(\sqrt{g}/\log g)$. 
\end{corollary}

\subsection{Probabilistic preliminaries}

The following lemma is widely known as Chernoff's inequality, see Corollary 2.3 in \cite{JLR} or Theorem 4.4 in \cite{MU}.

\begin{lemma}\label{lem:chernoff}
Let $X$ be a binomial random variable $\mathrm{Bin}(n,p)$, and denote $\mu = \mathbb E[X]$. We have

\begin{align*}
    & \mathbb P(X\ge (1+\delta)\mu)\le \exp\left(-\tfrac{\delta^2 \mu}{2+\delta}\right) \text{ for every } \delta\ge 0,\\
    & \mathbb P(X\le (1-\delta)\mu)\le \exp\left(-\tfrac{\delta^2\mu}{2}\right) \text{ for every } \delta\in [0,1].
\end{align*}
\end{lemma}

A direct application of Chernoff's inequality is that the number of edges in $G\in \mathcal G(n,p)$ is highly concentrated. This
can be used to deduce the following.

\begin{corollary}\label{cor:LB RG}
For every $\varepsilon\in (0,1)$ and every $p = p(n) > 0$ such that
$np\in [1-\varepsilon, n^{1-\varepsilon}]$, the binomial random graph $G\in
\mathcal G(n,p)$ has local boxicity at least $\tfrac{\varepsilon}{41} \cdot np$ asymptotically almost surely.
\end{corollary}
\begin{proof}
By~\cref{lem:chernoff} we know that the number of edges of $G$ is
a.a.s.\ between $n^2p/3$ and $n^2p$. Let us condition on the random variable $m$ and on the a.a.s.\ event that $m\in
[n^2p/3, n^2p]$. Then, the random graph $G\in \mathcal G(n,m)$ has a uniform distribution among all graphs with $n$ vertices and $m$ edges. On the other hand, the number of $n$-vertex graphs with $m$ edges is at least
$\exp(m \ln(n(n-1)/2m))\ge \exp(\varepsilon n^2p\ln n/4)$ for every large enough
$n$. By~\cref{obs:counting}, the number of $n$-vertex graphs graphs of boxicity
less than $\varepsilon np/41$ is at most 
\[
   \exp\left(n\cdot \tfrac{\varepsilon np}{41}(3\ln n + 7\ln(np))\right)=o\left(\exp(\varepsilon n^2p\ln n/4)\right).
\]
Thus, only a negligible proportion of
all graphs on $n$ vertices and $m$ edges have boxicity less than $\varepsilon np/41$ a.a.s.\ It follows
that when $np\in [1-\varepsilon, n^{1-\varepsilon}]$, $G$ has local boxicity at least $\varepsilon np/41$ a.a.s.\
\end{proof}

The next lemma is widely known under the name Lov\'asz Local Lemma, see \cite{EL}.

\begin{lemma}[\cite{EL}, page 616]\label{lem:LLL}
Let $G_D$ be a graph with vertex set $[n]$ and maximum
degree $d$, and let $A_1,\dots, A_n$ be events defined on some probability space such that for each $i\in [n],
\mathbb P(A_i)\le 1/4d$. Suppose further that each $A_i$
is jointly independent of the events $(A_j)_{ij\not\in E(G_D)}$. Then, $\mathbb P(A^c_1\cap\dots \cap A^c_n) > 0$.
\end{lemma}

\subsection{Other combinatorial preliminaries}
For every $r,t,k,s\in \mathbb N, 2\le t\le k\le s$, a \emph{Steiner
  system}  with parameters $(t,k,s)$ is a
family $S_1, S_2, \dots, S_r $ of $k$-element subsets of $[s]$, called \emph{blocks},  such that every $t$-element
subset  of $[s]$ is contained in exactly one block. One may easily deduce that this implies $r=\binom{s}{t}/\binom{k}{t}$.

\smallskip

We will use Steiner systems to prove results on local boxicity with
the help of
the following observation (recall that the notation $G\langle S\rangle$
was introduced in Section~\ref{sec:prelbox}). 

\begin{lemma}\label{lem:groupslbox}
Fix a graph $G = (V, E)$ and integers $k,s\in \mathbb N$ such that $2\le k\le s$, and let $V_1, V_2, \dots, V_s$ be a partition of $V$. Let $r = \binom{s}{2}/\binom{k}{2}$ and $(S_1, S_2, \dots, S_r)$ be a Steiner system with parameters $(2,k,s)$. Then, 
\begin{equation*}
  \lbox(G)\le \tfrac{s-1}{k-1}\cdot \max_{1\le j\le r} \lbox\big(G[{\textstyle \bigcup}_{i\in S_j} V_i]\big).
\end{equation*}
\end{lemma}
\begin{proof}
We first prove that $G = \bigcap_{j\in [r]} G\langle \bigcup_{i\in S_j}
V_i\rangle$. Since all the graphs $G\langle \bigcup_{i\in S_j}
V_i\rangle$ are supergraphs of $G$, it suffices to show that every non-edge $uv$ in $G$
appears in at least one of the graphs $G\langle \bigcup_{i\in S_j}
V_i\rangle$, and thus in at least one of the graphs $G[\bigcup_{i\in S_j}
V_i]$. Note that $\{ u,v\}$ is a subset of $V_i\cup V_j$,
for some pair $i,j$, and by definition of Steiner system $\{i,j\}$ is
a subset of some block $S_\ell$. It follows that
$u,v\in \bigcup_{i\in S_\ell} V_i$, as desired.

Recall that by \cref{obs:angle}, $\lbox(G\langle S\rangle)=
\lbox(G[S])$ for any subset $S$ of vertices of $G$, since vertices of
$G-S$ can be mapped to $\mathbb{R}$ in every dimension. This implies 
that, in addition,  there is a $\lbox(G[S])$-local box representation
of $G\langle S\rangle$ in which all the boxes of the vertices of $G-S$
are 0-local.

A simple property of Steiner systems is that every element of $[s]$ is
contained in exactly $\tfrac{s-1}{k-1}$ blocks, and thus every set
among $(V_1, V_2,\dots, V_s)$ participates in exactly
$\tfrac{s-1}{k-1}$ graphs $G[\bigcup_{i\in S_j}
V_i]$, $1\le j\le r$. Hence, we conclude that
\begin{equation*}
    \lbox(G)\le \tfrac{s-1}{k-1} \cdot \max_{1\le j\le r} \lbox\big(G[{\textstyle \bigcup}_{i\in S_j} V_i]\big),
  \end{equation*}
as desired.
\end{proof}

To be able to use Steiner systems we will need to following well known construction.
For every prime number $q$, the \emph{affine plane} over $\mathbb F_q$ is a geometric object consisting of $q^2$ points and $q^2+q$ lines such that:
\begin{itemize}
    \item every line contains $q$ points,
    \item every point is contained in $q+1$ lines, and
    \item every pair of points is contained in exactly one line\footnote{Indeed, the affine plane over $\mathbb F_q$ may be defined for every $q$ that is a power of a prime number. Although we do not give a precise definition of this object, we state its properties, which will be of interest for us.}.
\end{itemize}

\begin{observation}\label{obs:affine}
For every prime number $q\in \mathbb N$, the affine plane over $\mathbb F_q$ is a Steiner system with parameters $(2, q, q^2)$.
\end{observation}

Because the affine plane only exists for specific values of $q$, we will need the following result of Dusart, see~\cite{DusPHD}. Stronger results were established in the sequel by the same author in \cite{Dus1, Dus} and by Baker, Herman and Pintz in~\cite{BHP}.

\begin{theorem}[\cite{DusPHD}, Theorem 1.9]\label{thm:dus}
For every real number $t\ge 3275$ there is a prime number in the interval $\left[t, t+\tfrac{t}{2\ln^2 t}\right]$.\qed
\end{theorem}

\begin{corollary}\label{cor:dus}
For every real number $t\ge 3275^2$ there is a square of a prime number in the interval $\left[t, t+\tfrac{7t}{\ln^2 t}\right]$.
\end{corollary}
\begin{proof}
Let $k\in \mathbb N$ be the least integer such that $t\le k^2$. Then, we have that $k^2\in [t, t + 2\sqrt{t}+1]$ and by~\cref{thm:dus} there is a prime number $q$ in the interval $\left[k, k + \tfrac{k}{2\ln^2 k}\right]$. We conclude that $q^2$ is in the interval $\left[k^2, k^2 + \tfrac{11k^2}{10\ln^2 k}\right]$. Since $t\le k^2$ and $k^2 + \tfrac{11k^2}{10\ln^2 k}\le t + 2\sqrt{t} + 1 + \tfrac{11(t+2\sqrt{t}+1)}{10\ln^2(\sqrt{t})}\le t+\tfrac{7t}{\ln^2 t}$ for every real number $t\ge 3275^2$, the claim is proved.
\end{proof}

\section{Local boxicity and maximum degree}\label{sec:deg}

For every $\Delta\in \mathbb N$ with $\Delta/\ln \Delta\ge 3275^2$, fix a prime number $q = q(\Delta)$ with 
\begin{equation*}
    q^2\in \left[\tfrac{\Delta}{\ln \Delta}, \tfrac{\Delta}{\ln \Delta}\left(1+\tfrac{7}{\ln^2(\Delta/\ln \Delta)}\right)\right]
\end{equation*}
(such a prime number $q$ exists by \cref{cor:dus}). Let $\mathcal S_q
= (S_1, S_2, \dots, S_r)$ be the Steiner system over $[q^2]$ with
parameters $(2, q, q^2)$ given by~\cref{obs:affine} (so in particular
each of the sets $S_1, S_2, \dots, S_r$ has size $q$ and $r =
\binom{q^2}{2}/\binom{q}{2} = q(q+1)$).

We start by proving the following lemma.

\begin{lemma}\label{lem:partition}
Consider an integer $\Delta$ with $\Delta/\ln \Delta\ge 3275^2$, and
let $q=q(\Delta)$ and $\mathcal S_q
= (S_1, S_2, \dots, S_r)$ be as defined above. For every $\Delta$-regular graph $G$ one may partition the vertices of $G$ into $q^2$ sets $V_1, V_2, \dots, V_{q^2}$ so that for every $j\in [r]$, the graph $G[\bigcup_{i\in S_j} V_i]$ has maximum degree at most $\left(1+4\sqrt{\tfrac{q\ln \Delta}{\Delta}}\right)\tfrac{\Delta}{q}$.
\end{lemma}

\begin{proof}
Let us color the vertices of $G$ uniformly at random and
independently with the  colors $1,2,\dots,q^2$. Fix a vertex $v\in
G$. For every color $i\in [q^2]$, let $X_i$ be the number of
neighbors of $v$ in color $i$. Lemma~\ref{lem:chernoff}
(Chernoff's inequality) directly implies that for every vertex $v\in V(G)$, $j\in [r]$ and $\delta\in [0,1]$,
\begin{equation}\label{eq:1}
\mathbb P\bigg(\sum_{i\in S_j} X_i\ge (1+\delta)\tfrac{\Delta}{q}\bigg)\le \exp\left(-\tfrac{\delta^2 \Delta}{3q}\right).
\end{equation}
For each color $i\in [q^2]$, let $V_i$ be the set of vertices of $G$
colored in $i$. This produces a partition of the vertex set of $G$ into $q^2$ color classes $V_1, V_2, \dots, V_{q^2}$. Let $v$ be a vertex of $G$ and let $V_k$ be the set containing $v$ for some $k\in [q^2]$. For every $j\in [r]$, let $A_{j,v}$ be the event that the number of neighbors of $v$ in the set $\bigcup_{i\in S_j} V_i$ is at least $\big(1+4\sqrt{\tfrac{q\ln \Delta}{\Delta}}\big)\tfrac{\Delta}{q}$. By~\eqref{eq:1} the probability for this event is at most $\Delta^{-16/3}$.\par

Now, for every $j\in [r]$ and every $v\in V(G)$, the event $A_{j,v}$
is independent from the family of events $(A_{j',v'})_{j'\in [r], v':
  d_G(v',v)\ge 3}$. The number of events that remain is less than
$r(1+\Delta+\Delta^2)\le 2q(q+1)\Delta^2\le 4\Delta^3$. One may
conclude that for every $\Delta$ in the range given by the lemma, the
assumptions of \cref{lem:LLL} (the Lov\'asz Local Lemma) are satisfied
for the dependency graph of the events $(A_{j,v})_{j\in [r], v\in
  V(G)}$ with vertices $(A_{j,v})_{j\in [r], v\in V(G)}$ and edges
$A_{j,v}A_{j',v'}$ for every two pairs $(j,v)$ and $(j',v')$ such that
$A_{j,v}$ and $A_{j',v'}$ are not independent. We conclude that the
event $\bigcap_{j\in [r], v\in V(G)} A^c_{j,v}$ happens with positive
probability, which concludes the proof.
\end{proof}

For any real number $\Delta$, let $\lbox(\Delta)$ be the maximum local
boxicity of a graph of maximum degree at most $\Delta$. 

\begin{corollary}\label{cor:lboxdelta}
For any $\Delta\in \mathbb N$ with $\Delta/\ln \Delta\ge 3275^2$, let
$q=q(\Delta)$ be as defined above. Then we have 
$$\lbox(\Delta)\le (q+1)\cdot \lbox\Big(\big(1+4\sqrt{\tfrac{q\ln \Delta}{\Delta}}\big)\tfrac{\Delta}{q}\Big).$$
\end{corollary}
\begin{proof}
 As every graph of maximum degree at most $\Delta$ is an induced
 subgraph of some $\Delta$-regular graph  (see for instance Section
 1.5 in~\cite{MR13}) and local boxicity is monotone under taking
 induced subgraph, it is enough to prove that every $\Delta$-regular
 graph has local boxicity at most $(q+1)\cdot \lbox\big(\big(1+4\sqrt{\tfrac{q\ln \Delta}{\Delta}}\big)\tfrac{\Delta}{q}\big)$.
 
 This is a direct consequence of~\cref{lem:groupslbox} (with $k=q$ and
 $s=q^2$), combined with~\cref{lem:partition}.
\end{proof}

We are now ready to prove Theorem~\ref{thm:main}. 

\begin{proof}[Proof of Theorem~\ref{thm:main}]
  By \cref{cor:delta} it is enough to prove the upper bound.

Define the function
\begin{equation*}
\alpha:t\in [2, +\infty) \mapsto \prod_{i\ge 1} \left(1+\tfrac{18}{\ln^2(t^{3^i/2^i})}\right)^{-1} = \prod_{i\ge 1} \left(1+\tfrac{18\cdot 4^i}{9^i \ln^2 t}\right)^{-1}\in (0,1).
\end{equation*}

Note that $\alpha$ is well defined since for every fixed $t\ge 2$ we have
\begin{equation*}
\ln\left(\prod_{i\ge 1} \left(1+\tfrac{18\cdot 4^i}{9^i \ln^2 t}\right)\right) = \sum_{i\ge 1} (1+o_i(1)) \tfrac{18\cdot 4^i}{9^i \ln^2 t},
\end{equation*}
which converges to some real number. Note that $\alpha$ is an
increasing function and for every $\Delta\ge 2$ we have
$\alpha(\Delta) < 1$ (since $\alpha(\Delta)$ is defined as a
product of factors in the interval $(0,1)$). Moreover,  for every $\Delta\ge 4$,
\[
  \left(1+\tfrac{18}{\ln^2
      \Delta}\right)\alpha(\Delta^{2/3})=\left(1+\tfrac{18}{\ln^2
      \Delta}\right)\prod_{i\ge 1} \left(1+\tfrac{18}{\ln^2(\Delta^{3^{i-1}/2^{i-1}})}\right)^{-1} 
  =\alpha(\Delta) .
  \]

  We now prove that
 \begin{equation}\tag{$*$}\label{eq:ind}
    \text{there is a constant } C\ge 1 \text{ such that for every  }\Delta \ge 2, \lbox(\Delta)\le
    C\alpha(\Delta)\cdot \Delta.
  \end{equation}
  
We argue by induction on $\Delta$. First, using any existing bound on
the (local) boxicity of graphs of bounded degree, for any fixed
$\Delta_0$ (to be chosen later) one may find $C\ge 1$ such that for every
$\Delta \in[2, \Delta_0]$, $\lbox(\Delta)\le C\alpha(\Delta)\cdot \Delta$. 
Now, suppose that for some $\Delta_1\ge \Delta_0$, \eqref{eq:ind} holds for every $\Delta\le \Delta_1-1$. We prove \eqref{eq:ind}  for $\Delta = \Delta_1$. By choosing $\Delta_0$ so that $\Delta_0/\ln \Delta_0\ge 3275^2$, we have by~\cref{cor:lboxdelta} that 
\begin{equation}\label{eq:2}
\lbox(\Delta)\le (q+1)\cdot\lbox\left(\left(1+4\sqrt{\tfrac{q\ln
        \Delta}{\Delta}}\right)\tfrac{\Delta}{q}\right),
\end{equation}

with $q^2\in \left[\tfrac{\Delta}{\ln \Delta}, \tfrac{\Delta}{\ln \Delta}\left(1+\tfrac{7}{\ln^2(\Delta/\ln \Delta)}\right)\right]$.
Notice that
\begin{equation}\label{eq:2.5}
    \tfrac{\Delta}{q}\le \sqrt{\Delta \ln \Delta}
\end{equation}
and
\begin{equation}\label{eq:3}
q+1\le \left(\sqrt{\tfrac{\Delta}{\ln\Delta}\left(1+\tfrac{7}{\ln^2(\Delta/\ln \Delta)}\right)}\right) + 1\le \sqrt{\tfrac{\Delta}{\ln\Delta}}\left(1+\tfrac{7}{\ln^2(\Delta/\ln \Delta)}\right) + 1.
\end{equation}

Moreover, by choosing $\Delta_0$ large enough, we obtain that for every $\Delta\ge \Delta_0$
\begin{equation}\label{eq:4}
\sqrt{\tfrac{\Delta}{\ln\Delta}}\left(1+\tfrac{7}{\ln^2(\Delta/\ln \Delta)}\right) + 1\le \sqrt{\tfrac{\Delta}{\ln\Delta}}\left(1+\tfrac{8}{\ln^2\Delta}\right)+1\le \sqrt{\tfrac{\Delta}{\ln\Delta}}\left(1+\tfrac{9}{\ln^2\Delta}\right),
\end{equation}
and furthermore,
\begin{equation}\label{eq:5}
1+4\sqrt{\tfrac{q\ln \Delta}{\Delta}}\le 1+4\sqrt{2\sqrt{\tfrac{\ln \Delta}{\Delta}}}\le 1+\tfrac{8}{\ln^2 \Delta}.
\end{equation}

By combining \eqref{eq:2.5}, \eqref{eq:3},
\eqref{eq:4} and \eqref{eq:5} with the induction hypothesis (and using
the fact that $\alpha$ is an increasing function),  we obtain that for $\Delta_0$ large enough
\begin{eqnarray*}
\lbox(\Delta) & \le &C \left(1+\tfrac{8}{\ln^2
                      \Delta}\right)\left(1+\tfrac{9}{\ln^2
                      \Delta}\right)\alpha\left((1+4\sqrt{q\ln
                      \Delta/\Delta})\sqrt{\Delta\ln \Delta}\right)
                      \Delta \\ &\le& C \left(1+\tfrac{18}{\ln^2 \Delta}\right)\alpha(\Delta^{2/3}) \Delta=C\alpha(\Delta)\Delta,
\end{eqnarray*}
which completes the proof of \eqref{eq:ind}. Since $\alpha(\Delta)\le 1$ for any $\Delta\ge 2$ and
$\lbox(1)=1$ we obtain that $\lbox(\Delta)\le C \Delta$ for any
$\Delta\ge 1$, as desired.
\end{proof}

We now explain how to deduce \cref{thm:edges} from \cref{thm:main}.

\begin{proof}[Proof of \cref{thm:edges}]
  By \cref{cor:edges}, it suffices to prove the upper bound.
Let $S$ be the set of vertices of degree at least $\sqrt{m}$. Then
$G-S$  has maximum degree $\sqrt{m}$, and thus local boxicity
$O(\sqrt{m})$ by \cref{thm:main}. On the other hand we have $|S|\cdot
\sqrt{m}\le 2m$ and thus $|S|\le 2\sqrt{m}$. It then follows from
\cref{obs:g-v} that $G$ has local boxicity at most $\lbox(G-S)+2\sqrt{m}=O(\sqrt{m})$.
\end{proof}

\section{Proof of \texorpdfstring{\cref{thm:gnp}}{Theorem 1.3}}\label{sec:RG}

A \emph{multicyclic component} in a graph $G$ is a connected component
of $G$ that contains at least two cycles. We will need a
preliminary lemma, which may be found in a more general form as Lemma
2.10 in \cite{FK1}.

\begin{lemma}[Lemma 2.10 in \cite{FK1}]\label{lem 2.10}
Let $p = c/n$ with $c < 1$. Then, the probability that a graph $G\in \mathcal{G}(n,p)$ contains a multicyclic  component is at most $\tfrac{2}{(1-c)^3 n}$.
\end{lemma}

We are now ready to prove \cref{thm:gnp}.

\begin{proof}[Proof of \cref{thm:gnp}]
The lower bound was already proved in \cref{cor:LB RG}, so we
concentrate on the upper bound. 
If $np\gg \ln n$, then it is simple consequence of Chernoff's inequality that the maximum degree of the random graph $G\in \mathcal
G(n, p)$ is at most $2np$ a.a.s.\ (see for example Theorem 3.4 in
\cite{FK1}). Thus, in this regime, the upper bound follows directly from Theorem~\ref{thm:main}.

Fix an arbitrary small real
number $\varepsilon > 0$. We now concentrate on the regime $np\in [1-\varepsilon, \ln^2 n]$ (here we choose $\ln^2 n$ so to have $\ln n\ll \ln^2 n\ll n^{1/3}$). Partition the vertex set of $G$ in $\lceil 2(1+\varepsilon) np\rceil$ subsets uniformly at random (that is, every vertex is put into each of the $\lceil 2(1+\varepsilon) np\rceil$ subsets with the same probability, and independently from the other vertices). By Lemma~\ref{lem:chernoff}, the number of vertices in any of the $\lceil 2(1+\varepsilon) np\rceil$ subsets is at most $\tfrac{(1+\varepsilon/2) n}{\lceil 2(1+\varepsilon)np\rceil}$ with probability $1 - \exp(-\Theta(n/np)) = 1 - o(1/n)$. We condition on this event. Thus, the union of every two subsets induces a binomial random graph with parameters $n'\le 2\tfrac{(1+\varepsilon/2) n}{\lceil 2(1+\varepsilon) np\rceil}$ and
\[
    p \le \tfrac{1}{2(1+\varepsilon) n/\lceil 2(1+\varepsilon) np\rceil} = \tfrac{(1+\varepsilon/2)/(1+\varepsilon)}{2(1+\varepsilon/2) n/\lceil 2(1+\varepsilon) np\rceil} = \tfrac{(1+\varepsilon/2)/(1+\varepsilon)}{n'}.
\]

By Lemma~\ref{lem 2.10} for $n = n'$ and $c = (1+\varepsilon/2)/(1+\varepsilon)$ and a direct union bound we conclude that the probability that there is a multicyclic component in the graph induced by the union of any two of the above subsets of vertices is $O\left(\tfrac{(np)^2}{(n/np)}\right) = O(n^2p^3)$.

All in all, the probablity that the union of some two subsets of
vertices induces a graph with a multicyclic component is at most 
\[
\lceil 2(1+\varepsilon)np\rceil \cdot o\left(\tfrac{1}{n}\right) +
O(n^2p^3) \left(1-\lceil 2(1+\varepsilon)np\rceil \cdot o\left(\tfrac{1}{n}\right)\right) = o(1).
\]

Moreover, graphs without multicyclic  components have boxicity at most
two (and therefore, local boxicity at most two as well). The proof is
completed by~\cref{lem:groupslbox}, applied with $k=2$, $s=\lceil
2(1+\varepsilon) np\rceil$, and choosing $C_4>0$ such that $\lceil
2(1+\varepsilon) np\rceil \le C_4 np$ for all $np\in [1-\varepsilon, \ln^2 n]$ (noting that for
any integer $\ell$, the
edge-set of the complete graph $K_\ell$ forms a Steiner system with
parameters $(2,2,\ell)$).
\end{proof}

\section{Local boxicity of graphs of bounded genus}\label{sec:genus}


This section is dedicated to an analogue of the following result of
Scott and Wood~\cite{SW} for local boxicity.

\begin{theorem}[Theorem 16 in \cite{SW}]\label{SW Thm}
The boxicity of any graph with Euler genus at most $g$ is at most $(12+o(1))\sqrt{g\ln g}$.
\end{theorem}

For this we only need to modify a single step of their proof, where
they use the fact that $n$-vertex graphs have boxicity at most $n$. Instead, we
use the following result.

\begin{theorem}[Theorem 12 in \cite{MM}]\label{thm:MM}
Any graph on $n$ vertices has local boxicity at most $24\cdot \tfrac{n}{\log n}$.
\end{theorem}

For the sake of completeness, we recall the main steps
of the proof of Theorem 16 in \cite{SW}, and how we implement the modification.

\begin{proof}[Proof of \cref{thm:genus}]
By \cref{cor:genus} it is enough to prove the upper bound.
We closely follow the proof of Theorem 16 in \cite{SW}. The proofs of
the claims below can be found there.

\begin{claim}[\cite{SW}]
$G$ contains a set $X$ of at most $60 g$ vertices, for which $\boxi(G\setminus X)\le 5$.
\end{claim}

If $|X|\le 10^4$, then 
\[
    \lbox(G)\le \boxi(G)\le \boxi(G\setminus X)+|X|\le 5+10^4,
\]
and the theorem holds.\par

If $|X| > 10^4$, fix $G_1 = G\langle V\setminus X\rangle$. Let $Y$ be
the set of vertices in $V\setminus X$ with at most two neighbours in
$X$, and let $Z$ be the set of vertices in $V\setminus X$ with more
than two neighbours in $X$. Define $G_2 = G\langle X,Y\rangle$ and
$G_3 = G\langle X\cup Z\rangle$. Then $G=G_1\cap G_2\cap G_3$, and
thus $\lbox(G)\le \lbox(G_1)+\lbox(G_2)+\lbox(G_3)$.

\begin{claim}[\cite{SW}]
$\mathrm{box}(G_1) = \mathrm{box}(G[V\setminus X])\le 5$ and $\mathrm{box}(G_2)\le 48\ln\ln(1000 g)$.
\end{claim}

Denote $H = G[X\cup Z]$.

\begin{claim}[\cite{SW}]
For any $k\ge 7$, the vertex set $X\cup Z$ can be partitioned in two subsets, $A$ and $B$, such that $\boxi(H\langle A\rangle)\le (k+2)\lceil 2e\ln(3002g)\rceil$, $\boxi(H\langle A,B\rangle)\le 2+3(k+1)\ln(3002g)$ and $|B|\le \tfrac{6g}{k-6}$.\qed
\end{claim}

By \cref{obs:angle} and \cref{thm:MM} we have  $\lbox(H\langle B\rangle) =
\lbox(H[B])\le
24\tfrac{6g}{(k-6)\ln\left(\frac{6g}{k-6}\right)}$. By
\cref{obs:angle2}, we deduce that
\begin{align*}
    \lbox(G_3) 
    &=\hspace{0.3em} \lbox(H)\\
    &\le \hspace{0.3em} \lbox(H\langle A\rangle)+\lbox(H\langle A,B\rangle)+\lbox(H\langle B\rangle)\\
    &\le \hspace{0.3em} \boxi(H\langle A\rangle)+\boxi(H\langle A,B\rangle)+\lbox(H\langle B\rangle)\\
    &\le \hspace{0.3em} (1 + 2e(k+2)\ln(3002g)) + (2+3(k+1)\ln(3002g)) + 24\tfrac{6g}{(k-6)\ln\left(\frac{6g}{k-6}\right)}.
\end{align*}
Choosing $k = \left\lfloor \frac{4\sqrt{g}}{\ln g}\right\rfloor$, we obtain
\[
    \boxi(G_3)\le (4(2e+3)+24\cdot 6/4+o(1)) \sqrt{g} < (70+o(1))\sqrt{g}.
\]

We deduce that 
\[
\lbox(G)\le  \boxi(G_1)+\boxi(G_2)+(70+o(1))\sqrt{g} = (70+o(1))\sqrt{g}.
\]
This concludes the proof.
\end{proof}

\section{Proofs of \texorpdfstring{\cref{thm:avgdeg}}{Theorem 1.5} and \texorpdfstring{\cref{thm:gcreg}}{Theorem 1.6}}\label{sec:gc}

We start with the following simple observation.

\begin{observation}\label{obs:g5}
Every co-interval graph of girth at least five is a forest in which
each connected component has diameter at most 3. On the other hand, every tree of diameter at most 3 is a co-interval graph.
\end{observation}
\begin{proof}
Denote by $P_4$ the path of length four and by $C_5$ the cycle of
length five. Then, $P^c_4$ contains an induced cycle of length four
and $C^c_5$ contains an induced cycle of length five. This shows that
$P^c_4$ and $C^c_5$ are not chordal, and thus not interval
graphs. Since interval graphs are closed under taking induced
subgraphs, this proves the first part of the claim.\par

Any tree $T$ with diameter at most 3 that contains at least one edge consists of two adjacent vertices $u,v$, together with a set $S_u$ of leaves adjacent to $u$ and a set $S_v$ of leaves adjacent to $v$ (where the two sets are possibly empty. Mapping $u$ to $\{0\}$, $v$ to $\{2\}$, all the elements of $S_u$ to $[1,2]$, and all the elements of $S_v$ to $[0,1]$, we obtain an interval representation of $T^c$, as desired. This proves the second part of the claim.
\end{proof}

\begin{proof}[Proof~\cref{thm:avgdeg}]
  Let $d=\ad(G^c)$ and $\ell=\lbox(G)$. By definition of local
  boxicity and~\cref{obs:g5}, the edge-set of $G^c$
  can be covered by a collection
  of trees such that each vertex is contained in at most $\ell$
of these trees. It follows that the edge-set of $G^c$
  can be partitioned into a collection
  of trees such that each vertex is contained in at most $\ell$
of these trees. Each tree can be oriented such that each vertex has
out-degree at most 1, while exactly one vertex has out-degree 0. It
follows that $G^c$ itself has an orientation where each vertex has out-degree
at most $\ell$, while at least one vertex has out-degree less than
$\ell$. Hence, $G^c$ has average degree less than $2\ell$.  It follows
that $\ell>d/2$, and since
$\ell$ is an integer, $\ell\ge \lfloor d/2 +1\rfloor$.
\end{proof}

\begin{lemma}\label{lem:kreglb}
Let $G$ be a graph and $k\in \mathbb N$ be odd. Suppose that $G^c$ is a $k$-regular graph with girth at least 5 that does not contain a perfect matching. Then, 
\begin{equation*}
   \lbox(G)\ge \tfrac{k+3}{2}.
\end{equation*}
\end{lemma}
\begin{proof}
Let $n=|V(G)|$. As every 1-regular graph has a perfect matching, we can
assume that $k\ge 3$. Assume for the sake of contradiction that $\lbox(G)<
\tfrac{k+3}2$. Then, $\lbox(G)\le \tfrac{k+1}2$. By definition of local
  boxicity and~\cref{obs:g5}, the edge-set of $G^c$
  can be partitioned into a family $\mathcal T$ of trees of diameter
  at most 3, such that each vertex is contained in at most $\tfrac{k+1}2$
of the trees. In each tree $T\in \mathcal{T}$, pick a vertex $r_T$ which is
not a leaf in $T$
(if no such vertex exists, pick any vertex $r_T$ of $T$), root
$T$ at $r_T$, and orient each edge of $T$ towards the root. Since $T$ is
a tree of diameter at most 3, this orientation of $T$ has the property
that each vertex has out-degree at most 1, and there is at most one
vertex $u$ distinct from the root that has
in-degree at least 1 in $T$. Moreover if such a vertex $u$ exists, it
is adjacent to the root (see \cref{fig:diam3}). This orientation of each tree $T\in
\mathcal{T}$ yields an orientation of $G^c$ with out-degree at most $\tfrac{k+1}2$. Let us denote by $D^+$ the
set of vertices of out-degree $\tfrac{k+1}2$, and define
$D^-=V(G)\setminus D^+$. For each vertex $v$ of $G$, let
$\delta(v)=d^+(v)-d^-(v)$, where $d^+(v)$ and $d^-(v)$ denote respectively the
out-degree and in-degree of $v$ in the orientation of $G^c$ defined
above. Note that vertices $v\in D^+$ satisfy $\delta(v)=1$, while
vertices $v\in D^-$ satisfy $\delta(v)\le -1$. Since $\sum_{v\in
  V(G)}\delta(v)=0$, it follows that $|D^+|\ge |D^-|$ and thus
$|D^+|\ge \tfrac{n}2$ and $|D^-|\le \tfrac{n}2$.

\medskip

\begin{figure}[htb]
 \centering
 \includegraphics{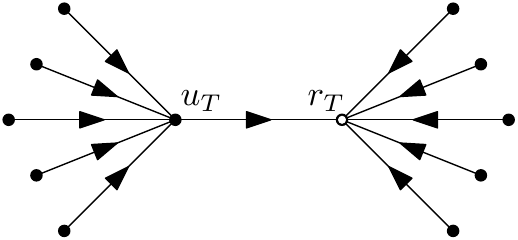}
 \caption{An orientation of a tree of diameter at most 3 towards a
   root $r_T$, which is not a leaf. The vertex $u_T$ is the only vertex
 distinct from $r_T$ with in-degree at least 1.}
 \label{fig:diam3}
\end{figure}

Consider some vertex $v$ which is a root in $i\ge 1$
  trees of $\mathcal{T}$. Then $d^+(v)\le \tfrac{k+1}2-i$ and thus
  $\delta(v)=2d^+(v)-k\le 1-2i\le -1$, with equality only if
  $i=1$. Hence,   $\sum_{v\in
  D^-}\delta(v)\le -|\mathcal{T}|$, with equality only if
  the roots $r_T$ are all distinct. Since $\sum_{v\in
  V(G)}\delta(v)=0$ and $\sum_{v\in
  D^+}\delta(v)=|D^+|$, we obtain $ |\mathcal{T}| \le |D^+|$, with
equality only if the roots $r_T$ are pairwise distinct.

Each vertex $v\in D^+$ has positive in-degree, and
thus $v$ is a non-root with in-degree at least 1 in at least one tree of
$\mathcal{T}$. By the property
of our orientation mentioned above (using the crucial fact that trees of
$\mathcal{T}$ have diameter at most 3), recall that each tree $T$ of $\mathcal{T}$ contains at most one vertex of $D^+$ with in-degree at least
one in $T$. It follows that $|\mathcal{T}|\ge |D^+|$, with equality
only if each tree $T$ of $\mathcal{T}$ contains a unique vertex of $D^+$ with in-degree at least
one in $T$. This shows that
$|\mathcal{T}|=|D^+|$ and the equality implies that the roots $(r_T)_{T\in \mathcal T}$ all distinct and each tree $T$ of $\mathcal{T}$ contains a different vertex of $D^+$ with in-degree at least
one in $T$. Therefore, the function  $T\in
\mathcal{T}\mapsto u_T\in D^+$, where
$u_T$ is the unique vertex of $D^+$ with in-degree at least 1 in
$T$, is a bijection. Note that for every $T\in \mathcal{T}$, $u_T$ and
$r_T$ are adjacent, and since the roots $r_T$ are pairwise distinct, the
set of edges between $u_T$ and $r_T$, for $T\in \mathcal{T}$, forms a
matching in $G^c$. Since $|\mathcal{T}|=|D^+|\ge \tfrac{n}2$,
this
matching is a perfect matching of $G^c$, a contradiction.
\end{proof}

We are now ready to prove \cref{thm:gcreg}.

\begin{proof}[Proof of~\ref{thm:gcreg}.]
By~\cref{thm:avgdeg} and~\cref{lem:kreglb} it is sufficient to prove
the upper bound.
Assume that $G^c$ has an orientation with out-degree at most $d$, for
some integer $d$.  For every vertex $v\in V(G^c)$, define $T_v$ to be the tree
with root $v$ containing all the edges directed towards $v$. Then,
$\bigcup_{v\in G^c} T_v = G^c$ and every vertex participates in
at most $d+1$ trees. By \cref{obs:g5}, each of these trees is a co-interval graph, and thus $G$ has local boxicity at most $d$.

Assume first that $k$ is even. Then, $G^c$ has an
Eulerian orientation, i.e.\ an
orientation in which each vertex has out-degree $k/2$ and in-degree
$k/2$, and the observation above implies that $G$ has local boxicity
at most $k/2+1$, as desired.

If $k$ is odd, then $G^c$ has an orientation in which each vertex has
out-degree at most $\tfrac{k+1}2$ (this can be seen by decomposing
greedily $G^c$ into a union of edge-disjoint cycles and a forest, and
finding an orientation with out-degree at most 1 of each of these
graphs). The observation above then implies that $G$ has local
boxicity at most $\tfrac{k+1}2+1=\tfrac{k+3}2$, which proves the
theorem when $G^c$ has no perfect matching.

Finally, assume that  $k$ is odd and $G^c$ contains some perfect
matching $M$. Then, the graph obtained from $G^c$ by removing all the
edges of $M$ 
has an Eulerian orientation. For each edge $uv\in M$, let $T_{uv}$ be
the tree consisting of the edge $uv$ together with the edges oriented
towards $u$, and the edges oriented towards $v$ (recall that $G^c$ has
girth at least 5, so these edges induce a tree). Then,
$\bigcup_{uv\in M} T_{uv} = G^c$, and every vertex participates in
$\tfrac{k-1}2+1=\tfrac{k+1}2$ trees. By \cref{obs:g5}, each of these trees is a co-interval graph, and thus $G$ has local boxicity at most $\tfrac{k+1}2$, as desired.
\end{proof}

\section{Proofs of \texorpdfstring{\cref{thm:col}}{Theorem 1.7} and \texorpdfstring{\cref{thm:col18}}{Theorem 1.8}}\label{sec:col}

We will need the following two results (for the first one the exact
constants are not given explicitly in \cite{CW} but can be easily deduced from
their proof).

\begin{theorem}[\cite{CW}]\label{thm:CW}
Every graph with boxicity at most two and clique number at most $r$ has chromatic number at most $320 r\log(2r)$.
\end{theorem}

\begin{theorem}[\cite{AG}]\label{thm:AG}
Every triangle-free graph with boxicity at most two has chromatic number at most 6.
\end{theorem}

We say that a graph of local boxicity at most 2 is of \emph{type
  $(1,1)$} if the graph has a 2-local box representation in which
each box is local in at most one dimension
distinct from the first dimension.

\begin{lemma}\label{lem:type11}
Let $G$ be a graph of type $(1,1)$ and clique number
$\omega(G)=r$. Then $\chi(G)\le 320 r^2\log(2r)$. Moreover, if $r=2$, then $\chi(G)\le 12$.
\end{lemma}
\begin{proof}
Consider a $d$-dimensional 2-local box representation $(B_v)_{v\in G}$ of $G$ such
that each box  is local in at most one
dimension
distinct from the first dimension. For each vertex $v$, let
$I_v$ be the interval obtained by projecting $B_v$ on the first
dimension.  Let $G_1$ be the interval supergraph of $G$ associated to
the intervals $(I_v)_{v\in G}$.  For every $2\le i \le d$, let $D_i$
be the set of vertices whose boxes are local in
dimension $i$. Let $D_1=V(G)-\bigcup_{2\le i \le d} D_i$. Observe that
$D_1,\ldots,D_d$ form a partition of the vertex-set of $G$. For any
$1\le i \le d$, let
$\mathcal{P}_i$ be the family of connected components of $G_1[D_i]$, and
let $\mathcal{P}=\bigcup_{1\le i \le d} \mathcal{P}_i$ (note that
$\mathcal{P}$ forms a partition of $V(G)$). 

For each set $S\in\mathcal{P}$, $G_1[S]$ is connected and thus the set
$I_S = \bigcup_{v\in S} I_v$ is an interval. Let $H$ be the interval graph with vertex set
$\mathcal{P}$ associated to the family of intervals $(I_S)_{S\in \mathcal{P}}$. We
claim that $H$ has clique number at most $r$. Indeed, suppose for the
sake of contradiction that $r+1$ of the intervals $(I_S)_{S\in
  \mathcal{P}}$ have non-empty intersection. Then there are $r+1$ vertices $v_1, v_2,\dots,v_{r+1}$ from different
sets of $\mathcal{P}$, for which $\bigcap_{1\le i\le r+1} I_{v_i} \neq
\varnothing$. Thus, the graph $G_1$ contains the complete graph on
$v_1,v_2,\dots,v_{r+1}$. Note that if there is an edge in $G_1$
between two vertices $u,v$ lying in different sets $S_u,S_v\in\mathcal{P}$, respectively, then the sets $S_u$ and $S_v$ do not
lie in the same set $D_i$, and thus $u$ and $v$ are adjacent in $G$. This shows that every edge of $G_1$ between vertices
from different sets of $\mathcal{P}$ also appears in $G$, which implies
that $G$ contains a clique of size $r+1$, a contradiction. Hence, $H$ has clique number at most $r$, and since it is an interval
graph, $\chi(H)\le r$. Fix an $r$-coloring of $H$, and for each $S\in
\mathcal{P}$, fix a coloring of $G[S]$ with at most $320 r\log(2r)$
colors (such a coloring exists by \cref{thm:CW} since $G[S]$ has
boxicity at most 2).

We now assign to each vertex $v\in G$ a pair of colors as follows. Let
$S\in \mathcal{P}$ be such that $v\in S$. Then, the first color assigned to $v$ is the color of
$S\in V(H)=\mathcal{P}$ in the
coloring of $H$, and the second color is the color of $v$ in
$G[S]$. The total number of colors assigned is at most $r\cdot
320 r\log(2r)=320 r^2\log(2r)$. It remains to prove that this is a proper
coloring of $G$. Let $uv$ be an edge of $G$. If $u$ and $v$ lie in the
same set $S\in \mathcal{P}$, then since we have chosen a proper
coloring of $G[S]$ the second colors of $u$ and $v$ are different. If $u$
and $v$ lie in different sets $S_u,S_v$ of $\mathcal{P}$, then since
$uv$ is also an edge of $G_1$, $S_u$
and $S_v$ are adjacent in $H$ and thus the first colors of $u$ and $v$
are different. This shows that $\chi(G)\le 320 r^2\log(2r)$.

\smallskip
Assume now that $r=2$. Then by
\cref{thm:AG} each graph $G[S]$ is 6-colorable, and thus $G$ itself
has a coloring with at most $2\cdot 6=12$ colors, as desired.
\end{proof}

We are now ready to prove~\cref{thm:col18} and \cref{thm:col}.

\begin{proof}[Proof of~\cref{thm:col18}]
Let $(B_v)_{v\in G}$ be a 2-local box representation of $G$ in some
dimension $d$. We say that a vertex $v$ is local in dimension $i$
if its $d$-box $B_v$ is local in dimension $i$. 
If $G$ has local boxicity at most 1, it is the complete
join\footnote{The \emph{complete join} of $k$ graphs $G_1,\dots G_k$
  is obtained from the disjoint union of $G_1,\dots G_k$ by adding all
possible edges between $G_i$ and $G_j$, for any $1\le i < j\le k$.} of interval graphs, and since $G$ is triangle-free, this implies
that $G$ itself is a (triangle-free) interval graph, and thus has
chromatic number at most 2. So we can assume that some vertex $v$ is
local in two different dimensions, say 1 and 2
without loss of generality. For any $1\le i\le d$, let $D_i$ be the set of
vertices of $G$ that are local in dimension
$i$. Note that $v\in D_1\cap D_2$, and $G[D_1\cap D_2]$ has boxicity
at most 2, and thus $\chi(G[D_1\cap D_2])\le 6$ by
Theorem~\ref{thm:AG}. This shows that if $V=D_1\cup D_2$,
$\chi(G[V\setminus (D_1\triangle D_2)]) \le 6$.

Note that all
vertices of $V\setminus(D_1\cup D_2)$ are neighbors of $v$, and
therefore form an independent set in $G$.
Since there is no pair of vertices $u\in D_1\cap D_2$ and $ w\in V\setminus (D_1\cup D_2)$ that
are local in a common dimension, all vertices of $D_1\cap
D_2$ are adjacent to all vertices of $V\setminus (D_1\cup D_2)$. In
particular, if $V\ne D_1\cup D_2$, the set $D_1\cap D_2$ is also an
independent set in $G$. In this
case we conclude that $\chi(G[V\setminus (D_1\triangle D_2)]) \le
2$.

\smallskip

If $D_2\setminus D_1$ is an independent set, since $G[D_1]$ is of type
$(1,1)$, it follows from Lemma~\ref{lem:type11} that $\chi(G)\le \chi(G[D_1]) +
\chi(G[D_2\setminus D_1]) + \chi(G[V\setminus (D_1\cup D_2)])\le 12 +
1 + 1 = 14$. The symmetric argument shows that if $D_1\setminus D_2$
is an independent set, then $\chi(G)\le 14$. Consequently, we can
assume that neither $D_1\setminus D_2 $ nor $D_2\setminus D_1$ is an
independent set, and in particular $G[D_1\setminus D_2]$ and $G[D_2\setminus D_1]$ both contain at least one edge. Fix $uw\in G[D_1\setminus D_2]$. We now consider two cases.

\begin{itemize}
    \item Assume first that $u$ and $w$ are local together in a third
      dimension, say dimension 3 without loss of generality. 
      Then, since $uw$ is not part of a triangle in $G$, $D_2\setminus
      D_1\subset D_3$. If $D_1\setminus D_2\not\subseteq D_3$, without
      loss of generality there is a vertex in $D_1\setminus (D_2\cup D_3)$, and this vertex is adjacent to
    every vertex in $D_2\cap D_3 = D_2\setminus D_1$. Thus,
    $D_2\setminus D_1$ is an independent set, which contradicts our
    assumption above. Hence, we can assume that $D_1\setminus
    D_2\subseteq D_3$, and thus $D_1\triangle D_2\subseteq D_3$. Since
    $G[D_3]$ is of type $(1,1)$ it follows from \cref{lem:type11} and
    the discusion above that 
    \begin{equation*}
        \chi(G) \le \chi(G[D_1\triangle D_2])+\chi(G[V\setminus (D_1\triangle
                       D_2)])\le 12 + 6 = 18.
    \end{equation*}
    
    \item  We can now assume without loss of generality that $u\in D_3$ and
      $w\in D_4$. Note that in this case $D_3\neq D_1$ and $D_4\neq D_1$.
      Then, since $uw$ is not part of a triangle in $G$, $D_2\setminus
      D_1\subset D_3\cup D_4$. Moreover, if $D_1\setminus
      D_2\not\subseteq D_3\cup D_4$, then there is a vertex $w\in
      D_1\setminus (D_2\cup D_3\cup D_4)$, which is connected to every
      vertex in $D_2\setminus D_1$, and therefore $D_2\setminus D_1$
      is an independent set, contradicting our assumption. Thus,
      $D_1\triangle D_2\subseteq D_3\cup D_4$. Note that if $D_3\cap
      D_2 = \varnothing$ or $D_4\cap D_2 = \varnothing$, then all
      vertices of $D_2\setminus D_1$ are adjacent to $u$, or all are
      adjacent to $w$, respectively. In both cases we obtain that $D_2\setminus D_1$
      is an independent set, a contradiction. This shows that $D_2\cap
      D_3\neq \varnothing$ and $D_2\cap D_4\neq \varnothing$. Since
      $G[D_1\cap D_3, D_2\cap D_4]$ and $G[D_1\cap D_4, D_2\cap D_3]$
      are complete bipartite graphs, each of $D_1\cap D_3, D_1\cap
      D_4, D_2\cap D_3$ and $D_2\cap D_4$ is an independent set. Since
       $D_1\triangle D_2\subseteq D_3\cup D_4$, this shows that
       $G[D_1\triangle D_2]$ is 4-colorable. Thus,
    \begin{equation*}
        \chi(G)\le \chi(G[D_1\triangle D_2])+\chi(G[V\setminus D_1\triangle D_2]) \le 4+6=10.
    \end{equation*}
\end{itemize}
This concludes the proof of the theorem.\end{proof}

In the proof of~\cref{thm:col} we made no effort to optimize the
multiplicative constant.

\begin{proof}[Proof of~\cref{thm:col}]
We will prove  by induction on $r\ge 2$ that any graph of local
boxicity at most 2 and clique number at most $r$ has chromatic number
at most $320 r^3\log (2r)$. The claim holds for $r=2$
by Theorem~\ref{thm:col18}. Suppose now that $r\ge 3$ and the property
holds for $r-1$. Let $G=(V,E)$ be a graph of local boxicity at most 2 and
clique number at most $r$, and let $(B_v)_{v\in G}$ be a 2-local box
representation of $G$ in some dimension. If $G$ has local boxicity
at most 1, then $G$ is the complete join of interval graphs, and thus
$\chi(G)=\omega(G)=r$, so we can assume that some box $B_v$ is local
in two dimensions, say in dimensions 1 and 2 without loss
of generality. For $i=1,2$, let $D_i$ be the set of vertices whose
boxes are local in dimension $i$ (and note that $v\in D_1\cap D_2$). As
all the vertices of $V \setminus (D_1\cup D_2)$ are neighbors of $v$,
the graph $G[V \setminus (D_1\cup D_2)]$ has clique number at most
$r-1$, and by the induction hypothesis it has a coloring with at most
$320 (r-1)^3\log(2(r-1))\le 320 (r-1)^3\log(2r)$ colors. Each of $G[D_1]$
and $G[D_2]$ is of type $(1,1)$ and thus both $G[D_1]$ and
$G[D_2\setminus D_1]$ have a coloring with at most
$320 r^2\log(2r)$ by \cref{lem:type11}. It follows that $G$ has a
coloring with at most
\[
320 (r-1)^3\log(2r)+2\cdot 320 r^2\log(2r)\le 320 r^3\log(2r)
\]
colors, as desired.
\end{proof}

\section{Conclusion and open problems}

Two natural problems are to close the gap between the lower bound of $\Omega(\sqrt{m}/\log m)$
and the upper bound of $O(\sqrt{m})$ for graphs with $m$ edges, and the lower
bound of $\Omega(\sqrt{g}/\log g)$
and the upper bound of $O(\sqrt{g})$ for graphs of Euler genus $g$.

\medskip

We have proved that the family of graphs of local boxicity at most 2
is $\chi$-bounded, which extends a classical result on graphs with
boxicity at most 2. 
It was proved in~\cite{Kar} that if $G$ is the complement of a graph
of boxicity at most 2, then $\chi(G)=O(\omega(G)\log \omega(G))$, and
thus the family of complements of graphs of boxicity at most 2 is
$\chi$-bounded. A natural question is whether this extends to the family of
complements of graphs of \emph{local boxicity} at most 2.

\begin{question}\label{qn:gc}
Is the family $\{G^c\,|\, \lbox(G)\le 2\}$ $\chi$-bounded?
\end{question}

It was observed by James Davies (personal communication) that the answer to this question is
negative. Given an integer $n\ge 2$, the \emph{shift graph} $S_n$ is
the graph whose vertices are the ordered pairs $(i,j)$ with $1\le i
<j\le n$, with an edge between $(i,j)$ and $(k,\ell)$ if and only if
$j=k$ or $\ell=i$. It can be checked that $S_n$ is triangle-free, and
Erd\H os and Hajnal~\cite{EH} proved that for any $n\ge 2$, $\chi(S_n)=\lceil
\log n\rceil $. James Davies noted that the complement of $S_n$ has
local boxicity at most 2: to see this, map each pair $(i,j)$ to the
$n$-dimensional box whose projection is equal to $\{0\}$ in dimension
$i$, $\{1\}$ in dimension
$j$, and $\mathbb{R}$ in the other dimensions. This shows that the
family of triangle-free graphs in $\{G^c\,|\, \lbox(G)\le 2\}$ has
unbounded chromatic number, and thus \cref{qn:gc} has a negative answer.

\medskip

\noindent {\bf Recent development.} After we made our manuscript public, the authors of \cite{MM} improved their bound on the local boxicity of graphs of maximum degree $\Delta$ to $O(\Delta)$, matching the bound of \cref{thm:main} (and as a consequence, also matching the bound of \cref{thm:edges}). Their proof avoids the use of number theoretic tools altogether.

\medskip

\begin{acknowledgement}
The authors would like to thank Mat\v ej Stehl\'ik for interesting
discussions and for suggesting to investigate the chromatic number of
graphs of local boxicity at most 2 and their complements. We are also
grateful to James Davies for allowing us to present his negative
answer to \cref{qn:gc}, and to Bartosz Walczak for giving us the correct constant in \cref{thm:CW}.
\end{acknowledgement}


\end{document}